\documentclass[11pt,a4paper]{article}

\usepackage{amsmath, amssymb, amsthm}
\usepackage{geometry}
\usepackage{enumitem}
\usepackage{cite}

\usepackage[hidelinks]{hyperref}
\newtheorem{theorem}{Theorem}
\newtheorem{conjecture}{Conjecture}
\newtheorem{lemma}[theorem]{Lemma}

\newtheorem{corollary}[theorem]{Corollary}
\theoremstyle{definition}

\theoremstyle{remark}

\DeclareMathOperator{\lcm}{lcm}

\title{The lonely runner conjecture holds for eight runners}
\author{Matthieu Rosenfeld\thanks{
 This research was funded, in whole or in part, by the French National Research Agency (ANR) under grant agreement No. ANR-24-CE48-3758-01. In accordance with the objective of open access dissemination, the author applies a Creative Commons Attribution (CC-BY) license to any accepted article or manuscript (AAM) resulting from this submission.
}\\
LIRMM, Univ Montpellier, CNRS, Montpellier, France
}
\date{\today}

\begin{document}

\maketitle
\begin{abstract}
We prove that the lonely runner conjecture holds for eight runners. Our proof relies on a computer verification and on recent results that allow bounding the size of a minimal counterexample. We note that our approach also applies to the known cases with $4$, $5$, $6$, and $7$ runners. We expect that minor improvements to our approach could be enough to solve the cases of $9$ or $10$ runners.
\end{abstract}

\section{Introduction}

The Lonely Runner Conjecture is a well-known open problem in combinatorial number theory and Diophantine approximation. It was first introduced by Wills~\cite{phdWills1965,Wills1967Jun} with a purely number theoretic formulation and has since attracted significant attention. Cusick proposed another formulation of the problem in terms of view obstructions~\cite{Cusick1973}. The following runner interpretation  and current name of the conjecture is due to Goddyn \cite{Bienia1998Jan}. Consider $k+1$ runners with distinct constant speeds running around a unit-length circular track. The lonely runner conjecture asserts that for any runner there is a time at which the runner is at least a distance $1/(k+1)$ away from every other runner. The conjecture can be stated more formally as follows.
For any real $x$, $\|x\|$ is the distance from $x$ to the nearest integer.

\begin{conjecture}[Lonely Runner Conjecture]
For all integer $k \geq 1$, every set of distinct integers $v_1, \ldots, v_{k+1}$, and for all $i$, there exists a real number $t$ such that for every $j$, we have 
\[
  \|tv_i-tv_j\| \geq \frac{1}{k+1}.
\]
\end{conjecture}

Other formulations of the problem allow $v_1, \ldots, v_k$ to be real numbers, but Wills proved that the two conjectures are equivalent, and we can focus our attention on the integral case \cite{Wills1968Jun}. Moreover, adding the same constant to all speeds does not change the instance of the problem, so one may, without loss of generality, set the speed of one runner to $0$. Finally, if one only cares about the distance to $0$, then the speeds $x$ and $-x$ play symmetric roles, which implies the following equivalent formulation of the conjecture.

\begin{conjecture}[Lonely Runner Conjecture]
For all integer $k \geq 1$, and every set of distinct positive integers $v_1, \ldots, v_{k}$, there exists a real number $t$ such that for every $i$, we have 
\[
  \|tv_i\| \geq \frac{1}{k+1}.
\]
\end{conjecture}
We focus on this second formulation of the conjecture. In the remainder of this article, we say that a set $S\subseteq\mathbb{N}$ of size $k$ has the \emph{lonely runner property (abbreviated LR property)} if there exists $t$ such that for all $v\in S$, $\|tv\| \geq \frac{1}{k+1}.$

Over the years, the conjecture has been verified for small values of $k$. The case $k=2$ is trivial, and $k=3$ was settled by Betke and Wills~\cite{Betke1972Jun} and Cusick~\cite{Cusick1973, Cusick1974Jan, Cusick1982Jan}. The case $k=4$ was first resolved with the aid of computers~\cite{Cusick1984Oct} and later simplified~\cite{Bienia1998Jan}. For $k=5$, a proof was given by Bohman, Holzman, and Kleitman~\cite{Bohman2001Feb} and later simplified by Renault~\cite{Renault2004Oct}. The case $k=6$ was established by Barajas and Serra~\cite{Barajas2008Mar}. For a recent survey on the lonely runner conjecture see~\cite{Perarnau2025Nov}. In this article, we propose to settle the case $k=7$.
\begin{theorem}\label{thm:8runners}
For all set of integers $\{v_1,\ldots, v_7\}$ of size $7$ there exists a real number $t$ such that for all $i$,
\[
\|tv_i\| \geq \frac{1}{8}.
\]
\end{theorem}

In 2019, Tao proved that there exists a computable constant $C>0$ such that if the lonely runner conjecture is true for all $k<k_0$, and if all $k_0$-tuples $v_1,\ldots, v_{k_0}$ with $\max_i v_i\le k_0^{C k_0^2}$ have the LR property, then the lonely conjecture holds for $k=k_0$ \cite{Tao2018Dec}. For any $k_0$, it is enough to check finitely many cases to verify the conjecture for all $k$ up to $k_0$. This result was recently improved by Malikiosis, Santos and Schymura who provided an explicit bound on the speeds that grows more slowly than the one provided by Tao \cite{Malikiosis2024Nov} (another similar but weaker bound was also proven by Giri and Kravitz~\cite{Giri2023Apr}). By itself this is not enough to verify the conjecture for $k=7$ (or even smaller values of $k$) by listing all cases since there are too many cases to consider. However, this result is key for our proof of the case $k=7$. Indeed, our strategy allows us to prove that if there is a counterexample, then this counterexample is large, which together with the aforementioned bound allows us to conclude that there is no counterexample.

Some partial results on the Lonely runner conjecture in the context of lacunary sequences were also obtained by different authors
\cite{BARAJAS20095687,CZERWINSKI20181301,DubickasLacunary,pandey2009note,RUZSA2002181}. A stronger version of these results can easily be deduced as a corollary of a recent result on lacunary sequences \cite{Rosenfeld2025Mar}. In Section \ref{sec:lacunary}, we present the context, and we provide the improved result.

\section{Upper bound on the product of the speeds in a minimal counterexample}
Let us first state precisely the result from Malikiosis, Santos and Schymura \cite{Malikiosis2024Nov}.
\begin{theorem}[{\cite[Theorem A]{Malikiosis2024Nov}}]\label{UpperBoundOnSum}
If the lonely runner conjecture holds for $k$, then it holds for $k+1$ for all $k$-tuple $v_1, \ldots, v_k$ such that $\gcd(v_1, \ldots, v_k)=1$, and
\[
\sum_{S\subseteq \{1,\ldots, k\}} \gcd(\{v_i: i\in S\}) >\binom{k+1}{2}^{k-1}.
\]
\end{theorem}
Observe that we do not lose any generality with the condition $\gcd(v_1, \ldots, v_k)=1$, since any instance of the lonely runner conjecture can be reduced to an equivalent instance by dividing all speeds by their greatest common divisor.
This result is one of the main ingredients for our proof of the case with $8$ runners (although the weaker bound by Giri and Kravitz would probably suffice as well~\cite{Giri2023Apr}). It is tempting to simply check all such subsets, but there are too many such $7$-tuples. Instead of using this result directly, we will use the following more convenient corollary.

\begin{corollary}\label{cor:UpperBoundOnProduct}
If the lonely runner conjecture holds for $k$, then it holds for $k+1$ for all $k$-tuple $v_1, \ldots, v_k$ such that $\gcd(v_1, \ldots, v_k)=1$, and
\[
\prod_{i\in \{1,\ldots, k\}}v_i\ge  \left[\frac{\binom{k+1}{2}^{k-1}}{k}\right]^k\,.
\]
\end{corollary}
\begin{proof}
  This is a direct consequence of Theorem \ref{UpperBoundOnSum}, and of
\[
\sum_{S\subseteq  \{1,\ldots, k\}} \gcd(\{v_i: i\in S\})
>\sum_{i \in \{1,\ldots, k\}} v_i
\ge  k \left[\prod_{i\in \{1,\ldots, k\}}v_i\right]^{1/k}
\]
where the last inequality is simply the AM-GM inequality.
\end{proof}

The idea behind our proof is to prove that in any counterexample to the lonely runner conjecture for $k=7$ the product of the speeds is larger than the bound given by Corollary \ref{cor:UpperBoundOnProduct}, which would imply that there is no such counterexample. For this, we will prove for many primes $p$ that if there is a counterexample then the product of the speeds is divisible by $p$. If the product of all such primes $p$ is larger than the bound from Corollary \ref{cor:UpperBoundOnProduct}, then the lonely runner conjecture is proven for $k$.

\section{Finding a prime that divides the product of the speeds}
In this section, we provide the different lemmas that allow us to find prime divisors of any counterexample to the lonely runner conjecture. We begin with a simple lemma that yields a few factors of the product of the elements of a counterexample.
\begin{lemma}\label{smalldivisors}
  Let $k \geq 1$, and $\{v_1, \ldots, v_k\}$ be a set of $k$ distinct integers that does not have the LR property then
  \[
 \lcm(2,\ldots,k+1) \text{ divides } \prod_{i\in \{1,\ldots, k\}}v_i.
  \]
\end{lemma}
\begin{proof}
Suppose that no $v_i$ is divisible by $j$, for some $j\in \{2,\ldots,k+1\}$. Any such $k$-uple has the LR property, since for all $i$, 
\[
\left\|\frac{v_i}{j}\right\|\ge \frac{1}{j}\ge\frac{1}{k+1}\,.
\] 
That is, for all $j\in \{2,\ldots,k+1\}$, any $k$-uple without the LR property contains a speed divisible by $j$, which concludes this proof.
\end{proof}

Before establishing our main tool to find prime factors, we recall the following sufficient condition for a set of integers to have the LR property. 
A similar remark was already used in \cite{Bienia1998Jan}.
\begin{lemma}\label{StrongerGCDCondition}
Let $k \geq 3$ be an integer such that the lonely runner conjecture holds for $k-1$. Let $v_1, \ldots, v_k$ be a set of integers with $\gcd(\{v_1, \ldots, v_k\})=1$ but $\gcd(v_1,\ldots, v_{k-1})\not=1$. Then $v_1, \ldots, v_k$ has the LR property.
\end{lemma}
\begin{proof}
Let $q=\gcd(v_1,\ldots, v_{k-1})$. By assumption $q\not=1$ and $v_k$ is not divisible by $q$.

Let $t=a/b$, be a lonely runner solution for $v_1,\ldots,v_{k-1}$, that is, for all $i\in \{1,\ldots, k-1\}$, 
\[
\left\|v_i\frac{a}{b}\right\|\ge \frac{1}{k}>\frac{1}{k+1}\,.
\]
For all $i$, $v_ib$ is divisible by $qb$. Hence, for all $t\in \{qa, qa+b,\ldots, qa+(q-1)b\}$ and for all $i\in \{1,\ldots, k-1\},$
\[
\left\|v_i\frac{t}{qb}\right\|=\left\|v_i\frac{a}{b}\right\|>\frac{1}{k+1}\,.
\]
Now, for $v_k$, we have for all $j$,
\[
\left\|v_k\frac{qa+jb}{qb}\right\|
=\left\|v_k\frac{a}{b} + \frac{jv_k}{q}\right\|\,.
\]
But since, $\gcd(v_k,q)=1$, we have 
\[
\left\{\left\|\frac{jv_k}{q}\right\|: j\in\mathbb{N}\right\}=  \left\{\left\|\frac{j}{q}\right\|: j\in\mathbb{N}\right\}\,
\]
which implies
\[
\left\{\left\|v_k\frac{qa+jb}{qb}\right\|: j\in\mathbb{N}\right\}
=\left\{\left\|v_k\frac{a}{b} + \frac{j}{q}\right\|: j\in\mathbb{N}\right\}\,.
\]
This set contains $q$ points evenly spaced out in the unit circle. In particular, since $q\ge2$ 
at least one of these choices of $t\in \{qa, qa+b,\ldots qa+(q-1)b\}$ is such that
\[
\left\|v_k\frac{t}{qb}\right\|\ge\frac{1}{4}\ge\frac{1}{k+1}\,,
\]
since $k\ge3$. This concludes the proof.
\end{proof}
We are now ready to prove our main lemma.
\begin{lemma}\label{lem:needOfCovering}
Let $k\ge3$ be an integer such that the lonely runner conjecture holds for $k-1$. Let $p\in \mathbb{N}$ be a positive integer such that for all 
$v_1,\ldots, v_k\in \{0,\ldots, (k+1)p-1\}$ with
\begin{itemize}
\item for all subset $S\in \{v_1,\ldots, v_k\}$ of size $k-1$, we have $\gcd(S\cup\{(k+1)p\})=1$,
\item for all $i$, $v_i$ is not divisible by $p$,
\end{itemize} 
there exists $t\in \{0,\ldots,(k+1)p-1\}$ such that for all $i$,
\[
\left\|\frac{tv_i}{(k+1)p}\right\|\ge \frac{1}{k+1}\,.
\]
Then for any set of $k$ distinct integers $\{v_1, \ldots, v_k\}$ that does not have the LR property, we have 
  \[
 p \text{ divides } \prod_{i\in \{1,\ldots, k\}}v_i.
  \]
\end{lemma}
\begin{proof}
Let $p$ be as in the theorem statement and $\{w_1, \ldots, w_k\}$ be a set of distinct integers without the LR property. Without loss of generality, we can assume $\gcd(\{w_1, \ldots, w_k\})=1$, otherwise divide every $w_i$ by $\gcd(w_1, \ldots, w_k)$, this set does not have the LR property either and the product of the speed divides the previous product.

For all $i$, let $v_i$ be the remainder of the euclidean division of $w_i$ by $(k+1)p$. 

For all subset $S\in \{v_1,\ldots, v_k\}$ of size $k-1$, we have $\gcd(S\cup\{(k+1)p\})=1$. Indeed, suppose that $\gcd(\{v_1,\ldots, v_{k-1},(k+1)p\})=d\not=1$, then $\gcd(\{w_1,\ldots, w_{k-1}\})\not=1$ since it is divisible by $d$. By Lemma \ref{StrongerGCDCondition} and since $\gcd(\{w_1, \ldots, w_k\})=1$, it implies that $\{v_1, \ldots, v_k\}$ has the LR property which is a contradiction.

Finally, suppose for the sake of contradiction that $p$ does not divide any of the $w_i$. Then $p$ does not divide any of the $v_i$ either which implies by our theorem assumption that there exists $t\in \{0,\ldots,k+1\}$ such that for all $i$,
\[
\left\|\frac{tv_i}{(k+1)p}\right\|\ge \frac{1}{k+1}\,.
\]
But for all $i$, $w_i\equiv v_i \mod (k+1)p$, which implies for all $i$,
\[
\left\|\frac{tw_i}{(k+1)p}\right\|=\left\|\frac{tv_i}{(k+1)p}\right\|\ge \frac{1}{k+1}\,,
\]
which contradicts the fact that  $\{w_1, \ldots, w_k\}$ does not have the LR property. Hence, $p$ divides one of the $w_i$. This completes the proof.
\end{proof}

For given $k$ and $p$, we need to verify a property for roughly $\binom{(k+1)p}{k}$ $k$-uples to deduce that $p$ divides $\prod_{i\in \{1,\ldots, k\}}v_i$ for all $\{v_1, \ldots, v_k\}$ that does not have the LR property. It can be done with a computer.
For small values of $k$ (probably at least up to $k=6$), a naive implementation of this should be enough to apply our proof strategy. For $k=8$, we still exhaustively verify all $7$-uples with a backtracking algorithm, but we use a few branching techniques in the implementation to reduce the running time. We give some more details about the implementation in Section \ref{sec:implementation} (we also provide the implementation \cite{gitRepo}).

\section{Proof of the lonely runner conjecture for eight runners}
We rely on the fact that the lonely runner conjecture was already proven for 7 runners. We will see in the next section that our technique can also be used to prove the conjecture for $4$, $5$, $6$, and $7$ runners (assuming it holds for 2 runners).
\begin{proof}[Proof of Theorem \ref{thm:8runners}]
For the sake of contradiction suppose there exists a set $\{v_1,\ldots,v_7\}$ of $7$ distinct integers without the LR property. Without loss of generality, we may assume  $\gcd(\{v_1,\ldots,v_7\})=1$.
Let $\mathcal{P}=\prod_{i\in\{1,\ldots,7\}}v_i$.

By Corollary \ref{cor:UpperBoundOnProduct}, we have
\begin{equation}\label{eq:product8}
\mathcal{P} \le \left[\frac{\binom{8}{2}^{6}}{7}\right]^7 < 7,4\times 10^{54}\,.
\end{equation}

By Lemma \ref{smalldivisors}, $\mathcal{P}$ is divisible by $\lcm(\{2,3,4,5,6,7,8\})$.

Now using our computer program that verifies conditions of Lemma \ref{lem:needOfCovering}, we can verify that $\mathcal{P}$ is divisible by $p$, for all $p\in S$ where
 \begin{align*}
S=\{&31,37,43,47,53,59,61,67,71,73,79,83,89,97,101,103,\\
 &107,109,113,127,131,137,139,149,151,157,163\}\,.
 \end{align*}
 Hence, $\mathcal{P}$ is divisible by 
\[
 \lcm(\{2,3,4,5,6,7,8\}\cup S)=3\times5\times7\times8\times \prod_{p\in S}p\approx 1.82\times 10^{55}\,.
\]
Since $\mathcal{P}$ is a positive integer this contradicts \eqref{eq:product8}. Therefore, there does not exist such a set $\{v_1,\ldots,v_7\}$ of $7$ distinct integers without the LR property. This concludes the proof of the lonely runner conjecture for 8 runners.
\end{proof}

\section{The cases of 4, 5, 6 and 7 runners}
Unsurprisingly our proof can also be applied to the other known cases (technically this could certainly be applied to the case with 3 runners as well, but that would require writing a version of Lemma \ref{lem:needOfCovering} specifically for this case). This is the exact same idea, and for each case we simply provide in Table \ref{tab:parameters} the upper bound on the product, the list of primes for which we verified the conditions of Lemma \ref{lem:needOfCovering}, and the product of these primes.
\begin{table}[h]
    \footnotesize
    \begin{tabular}{|c|c|c|c|}\hline
         $k$ & $\left[\frac{\binom{k+1}{2}^{k-1}}{k}\right]^k$ & $S$ &  $\lcm(\{2,\ldots,k+1\}\cup S)$\\\hline\hline
         $3$ &  $1728$ & $\{7,11,13\}$& $12012$\\\hline
         $4$ &  $< 4\times 10^9$ & $\{17,19,23,29,31,37\}$& $>10^{10}$\\\hline
         $5$ &  $< 2\times10^{20}$ & $\{23,29,31,37,41,43,47,53,59,61,67,71\}$& $>10^{21}$\\\hline
         $6$ &  $< 10^{35}$ & $\{13, 19, 31, 37, 41, 43, 47, 53, 59, 61, 67, 71, 73, 79, 83, 89, 97, 101, 103\}$& $>2\times 10^{35}$\\\hline
    \end{tabular}
    \caption{The parameters used to apply the proof of Theorem \ref{thm:8runners} to the cases of $4$, $5$, $6$ and $7$ runners\label{tab:parameters}.}
\end{table}

Note that in the case $k=5$, we could have included $25$ and $32$ in $S$ (and $49$ for $k=6$). Indeed, in Lemma \ref{lem:needOfCovering} it is not required that $p$ be prime, and our implementation can be used to verify that  $k=5$ and $p=25$ respect the condition of the Lemma. However, for the sake of uniformity we decided only to use only primes here, as using powers of prime did not seem to help much for the case $k=7$ which is by far the most demanding in terms of computation.

\section{Implementation details}\label{sec:implementation}
The implementation is relatively straightforward, and can be found in the GIT repository associated with the article~\cite{gitRepo}. Before discussing the implementation we provide a reformulation of Lemma~\ref{lem:needOfCovering}.

Given $k$ and $p$, we say that $v\in \{1,\ldots,p-1\}$ \emph{covers} $j\in\mathbb{N}$ if 
$\left\|\frac{jv}{(k+1)p}\right\|< \frac{1}{k-1}$. Given $k$ and $p$, we say that a set $\{v_1,\ldots,v_i\}$ covers a set $S\subseteq\mathbb{N}$ if any $s\in S$ is covered by at least one of the $v_i$. Intuitively, if $\{v_1,\ldots,v_i\}$ does not cover $j$, then the time $\frac{j}{(k+1)p}$ is a suitable time to verify the LR property for $\{v_1,\ldots,v_i\}$.

For all positive integers $i,j,d$, $\left\|\frac{ij}{d}\right\|=\left\|\frac{(d-i)j}{d}\right\|$. This means that if $v_i$ covers $j$, then $v_i$ covers $(k+1)p-j$, and $(k+1)p-v_j$ covers $j$ and $(k+1)p-j$.
Using this notion of cover and this symmetries\footnote{Another obvious symmetry that we do not exploit is that if $i$ covers $j$, then $j$ covers $i$.}, Lemma \ref{lem:needOfCovering} can be reformulated equivalently as follows.

\begin{lemma}\label{lem:reformulated}
Let $k\in \mathbb{N}$ with $k\ge3$ such that the lonely runner conjecture holds for $k-1$ and $p\in \mathbb{N}$ be an integer.
Suppose that  there exists no $k$-uple $v_1,\ldots, v_k\in \{1,\ldots, \lfloor(k+1)p/2\rfloor\}\setminus p\mathbb{N}$ with
\begin{itemize}
\item for all subset $S\in \{v_1,\ldots, v_k\}$ of size $k-1$, we have $\gcd(S\cup\{(k+1)p\})=1$,
\item $v_1,\ldots, v_k$ covers $\{1,\ldots, \lfloor(k+1)p/2\rfloor\}$.
\end{itemize} 
Then the conclusion of Lemma \ref{lem:needOfCovering} holds, that is,  for any set of $k$ distinct integers $\{v_1, \ldots, v_k\}$ without the LR property, we have 
  \[
 p \text{ divides } \prod_{i\in \{1,\ldots, k\}}v_i.
  \]
\end{lemma}
With this formulation, verifying this property is similar to a set cover problem with the extra-constraint that some solutions are forbidden by the first condition. Our program implements basic backtracking techniques to exhaustively search for a "bad" cover for a selected pair $k$ and $p$. If we do not find such a cover we deduce that Lemma \ref{lem:reformulated} (and Lemma \ref{lem:needOfCovering}) can be applied to this choice of $k$ and $p$.

The first obvious and simple optimization is to precompute for each $i\in \{1,\ldots, \lfloor(k+1)p/2\rfloor\}$ the set of positions covered by $i$. Storing them as vectors of bits (\verb!bitsets! in C++), we can exploit the fact that processors are really efficient at doing multiple bitwise operations at once. 

The second optimization is to avoid considering all the $k$-uples by constructing $\{v_1,\ldots,v_k\}$ element by element and to cut branches as much as we can. That is, for any partial choice of the $v_i$ if we can then detect that there is no way to complete the choice of $v_i$ in a cover of  $\{1,\ldots, \lfloor(k+1)p/2\rfloor\}$ then we go to the next partial choice and do not even consider the completions of the current partial cover. To cut branches, we can simply compute at every iteration how many integers are not covered yet, how many new integers can be covered by the best remaining integer and how many elements are we allowed to add to the cover. We can also verify the first condition about subsets of size $k-1$ as soon as we have chosen the $k-1$ first elements.

Since $\{v_1,\ldots,v_k\}$ is taken as a set the order of the elements does not matter, which leads to two other simple optimizations. Whenever we finish the backtracking induced by a partial cover $\{v_1,\ldots,v_i, v_{i+1}\}$, we remember for all future $\{v_1,\ldots,v_i, v_{i+1}'\}$, that there is no reason to try to add $v_{i+1}$ (since we already know that no valid cover contains $\{v_1,\ldots,v_i, v_{i+1}\}$), so we eliminate it from the remaining available integers for the completions of the partial cover $\{v_1,\ldots,v_i\}$.

Lastly, for a given partial cover $\{v_1,\ldots,v_i\}$, we compute the integer $u$ that is covered by the smallest number of integers amongst the available integers. One of the remaining $v_j$ has to cover $u$, so we can require $u$ to be covered by $v_{i+1}$ without losing any solution. This minimizes the number of possible values that can be chosen for $v_{i+1}$ for a given choice of $\{v_1,\ldots,v_i\}$.

With all this, verifying the condition of Lemma \ref{lem:reformulated} for $k=7$ and $p=163$ takes roughly 32 hours with our implementation on a single processor core (on the machine used by the author).

\section{Possible further improvement} 
Based on how easy it was to find the list of primes once we had the right approach, we would conjecture that, ignoring the computational cost, this approach could probably be used for any number of runners. It could be the case that for any $k$, any large enough prime $p$ has the property needed to apply Lemma \ref{lem:needOfCovering}. For $k=3,4,5$ it seems to be true for all prime $p$ larger than $7$, $17$ and $23$ respectively. For $k=6$, it seems to be true for all prime $p$ larger than $31$ and for $13$ and $19$ (it is not true for $17$, $23$ and $29$). For $k=7$, it seems that any prime $p$ at least $31$ other than $41$ works. As already noted, it also works for at least some powers of primes (e.g., $k=5$ and $p\in\{25,32\}$).

On the other hand, the upper bound given by Corollary \ref{cor:UpperBoundOnProduct} grows fast as a function of $k$ and so the primes $p$ for which we need to verify Lemma \ref{lem:needOfCovering} grow fast as well. Our exhaustive search to verify the property does not scale well with the size of the primes, but improving different details in the approach might be enough to verify the Lonely Runner conjecture for larger values of $k$. In fact, simply replacing our ad-hoc backtracking algorithm implementation by state-of-the-art solvers might be enough for $9$ runners ($k=8$). It might also be possible to come up with variants of Lemma \ref{lem:needOfCovering}, that allow for instance to prove that at least two of the speeds are divisible by a given number $p$.

The first condition of Lemma \ref{lem:needOfCovering} is used to cut many branches. For $k=7$, one condition that could be added is that at most $k-3$ elements can be divisible by $4$ (we do not provide proof of this claim as it only cuts few branches and does not improve the running time). One might hope to find a collection of other stronger conditions that would actually reduce the search space.

There might even be a good provable reason for which Lemma \ref{lem:needOfCovering} holds for all $k$, for any large enough integer $p$. That could lead to a proof of the lonely runner conjecture.

\section{Improving the lacunary case}
\label{sec:lacunary}
In this context, an increasing sequence $(a_i)_{i\in\mathbb{N}}$ is said to be lacunary if there exists $\delta>0$ such that for all $i$, $\frac{a_{i+1}}{a_i}\ge 1+\delta$. If the set of speeds is sufficiently lacunary (that is, $1+\delta$ is large enough) then the lonely runner conjecture holds for this set, and there has been effort to lower the $1+\delta$ bound that implies the LR property.

Pandey showed that if for all $i$, $\frac{v_{i+1}}{v_i}\ge\frac{2(k+1)}{k-1}$, then the set $\{v_1,\ldots, v_k \}$ has the LR property \cite{pandey2009note}. Ruzsa, Tuza and Voigt had already proven a similar inequality \cite{RUZSA2002181}, and Barajas and Serra improved the constant $2(k+1)/(k-1)$ to $2$ \cite{BARAJAS20095687}. This was later improved by Dubickas, who reduced the constant to roughly $1+8e\frac{\log k}{k}$ in the following result (where $e\approx2.718$ is the base of the exponential function).

\begin{theorem}[{\cite[Corollary 1.2]{DubickasLacunary}}]
Suppose that $c$ is a constant strictly greater than $8e$. Then there is a positive integer $k(c)$ such that for each integer $k\ge k(c)$ and each set of $k$ positive numbers $v_1<v_2<\ldots<v_k$ satisfying
$\frac{v_{j+1}}{v_j}\ge 1+c\frac{\log k}{k}$ for all $j\in\{1,\ldots,n-1\}$ there is a positive number $t$ such that  $\|v_it\|\ge 1/(k+1)$ for each $i\in\{1,2,\ldots,k\}$.
\end{theorem}

This result was later slightly improved by Czerwiński who relaxed the lacunary condition to hold from $j$ starting at $\lfloor\frac{k+1}{8e}\rfloor$ instead of $j=1$ at the cost of increasing the constant $c$ \cite[Theorem 4]{CZERWINSKI20181301}.

We will see that as a corollary of the following result $c$ can in fact be lowered to $3e$ (the lacunary constant is lowered to $1+3e\frac{\log k}{k}$) and the condition that $k$ is large enough can be removed.
\begin{theorem}[{\cite[Theorem 9]{Rosenfeld2025Mar}}]
   Let $T$ be a set of positive real numbers such that
    \[
      | T \cap [2^{j}, 2^{j+1})| \le C
    \]
    for all integers $j$ and some constant $C$.
    Assume that for some integer $z\ge 3$ the inequality
    \begin{equation*}
      3eCz\le2^z \eqno(*)
    \end{equation*}
    holds.
    Then there exists a real $\theta\in[0,1]$ such that for all $t\in T$,
    \[
     \|\theta t\|\ge2^{-z}\,.
    \]
\end{theorem}
Choosing $z=\log_2(k+1)$, $C=\frac{k+1}{3e\log_2(k+1)}$ so the first condition becomes    
\[
    | T \cap [2^{j}, 2^{j+1})| \le \frac{k+1}{3e\log_2(k+1)}\,.
\]
Using the fact that 
\[
\left[(k+1)^{3e/(k+1)}\right]^{ \frac{k+1}{3e\log_2(k+1)}}=2\,,
\]
we deduce the following corollary.
\begin{corollary}
Let $(v_i)_{i\ge0}$ be a sequence of positive real numbers such that for all $i$, $\frac{v_{j+1}}{v_j}\ge (k+1)^{3e/(k+1)}$
Then there exists a real $\theta\in[0,1]$ such that for all $i$,
\[
  \|\theta v_i\|\ge\frac{1}{k+1}\,.
\]
\end{corollary}
Using the inequality $e^x>1+x$ for all $x>0$, we can replace our main condition by $\frac{v_{j+1}}{v_j}\ge 1+\frac{3e\log(k+1)}{k+1}$ as desired.

\bibliography{biblio}{}

\begin{thebibliography}{10}

\bibitem{Barajas2008Mar}
J.~Barajas and O.~Serra.
\newblock {The Lonely Runner with Seven Runners}.
\newblock {\em Electron. J. Combin.}, page R48, 2008.

\bibitem{BARAJAS20095687}
J.~Barajas and O.~Serra.
\newblock On the chromatic number of circulant graphs.
\newblock {\em Discrete Mathematics}, 309(18):5687--5696, 2009.
\newblock Combinatorics 2006, A Meeting in Celebration of Pavol Hell’s 60th Birthday (May 1–5, 2006).

\bibitem{Betke1972Jun}
U.~Betke and J.~M. Wills.
\newblock {Untere Schranken f{\ifmmode\ddot{u}\else\"{u}\fi}r zwei diophantische Approximations-Funktionen}.
\newblock {\em Monatsh. Math.}, 76(3):214--217, 1972.

\bibitem{Bienia1998Jan}
W.~Bienia, L.~Goddyn, P.~Gvozdjak, A.~Seb{\ifmmode\mbox{\H{o}}\else\H{o}\fi}, and M.~Tarsi.
\newblock {Flows, View Obstructions, and the Lonely Runner}.
\newblock {\em J. Combin. Theory Ser. B}, 72(1):1--9, 1998.

\bibitem{Bohman2001Feb}
T.~Bohman, R.~Holzman, and D.~Kleitman.
\newblock {Six Lonely Runners}.
\newblock {\em Electron. J. Combin.}, page~R3, 2001.

\bibitem{Cusick1973}
T.~W. Cusick.
\newblock {View-Obstruction Problems}.
\newblock {\em Aequationes Math.}, 9:165--170, 1973.

\bibitem{Cusick1974Jan}
T.~W. Cusick.
\newblock {View-obstruction problems in n-dimensional geometry}.
\newblock {\em J. Combin. Theory Ser. A}, 16(1):1--11, 1974.

\bibitem{Cusick1982Jan}
T.~W. Cusick.
\newblock View-obstruction problems. ii, 1982.

\bibitem{Cusick1984Oct}
T.~W. Cusick and C.~Pomerance.
\newblock {View-obstruction problems, III}.
\newblock {\em J. Number Theory}, 19(2):131--139, 1984.

\bibitem{CZERWINSKI20181301}
S.~Czerwiński.
\newblock The lonely runner problem for lacunary sequences.
\newblock {\em Discrete Mathematics}, 341(5):1301--1306, 2018.

\bibitem{DubickasLacunary}
A.~Dubickas.
\newblock The lonely runner problem for many runners,.
\newblock {\em Glas. Mat. Ser. III}, 46(66):25--30, 2011.

\bibitem{Giri2023Apr}
V.~Giri and N.~Kravitz.
\newblock {The structure of Lonely Runner spectra}.
\newblock {\em arXiv}, 2023.

\bibitem{Malikiosis2024Nov}
R.~D. Malikiosis, F.~Santos, and M.~Schymura.
\newblock {Linearly-exponential checking is enough for the Lonely Runner Conjecture and some of its variants}.
\newblock {\em arXiv}, 2024.

\bibitem{pandey2009note}
R.~K. Pandey.
\newblock A note on the lonely runner conjecture.
\newblock {\em J. Integer Seq}, 12(2511224):1233--11026, 2009.

\bibitem{Perarnau2025Nov}
G.~Perarnau and O.~Serra.
\newblock {The Lonely Runner Conjecture turns 60}.
\newblock {\em Computer Science Review}, 58:100798, 2025.

\bibitem{Renault2004Oct}
J.~Renault.
\newblock {View-obstruction: a shorter proof for 6 lonely runners}.
\newblock {\em Discrete Math.}, 287(1):93--101, 2004.

\bibitem{gitRepo}
M.~Rosenfeld.
\newblock The lonely runner conjecture.
\newblock \url{https://gite.lirmm.fr/mrosenfeld/the-lonely-runner-conjecture}, 2025.

\bibitem{Rosenfeld2025Mar}
M.~Rosenfeld and A.~Shen.
\newblock {Local obstructions in sequences revisited}.
\newblock {\em arXiv}, 2025.

\bibitem{RUZSA2002181}
I.~Ruzsa, Z.~Tuza, and M.~Voigt.
\newblock Distance graphs with finite chromatic number.
\newblock {\em Journal of Combinatorial Theory, Series B}, 85(1):181--187, 2002.

\bibitem{Tao2018Dec}
T.~Tao.
\newblock {Some remarks on the lonely runner conjecture}.
\newblock {\em Contrib. Discrete Math.}, 13(2), 2018.

\bibitem{phdWills1965}
J.~Wills.
\newblock {\em Zwei Probleme der inhomogenen diophantischen Approximation.}
\newblock PhD thesis, TU Berlin, 1965.

\bibitem{Wills1967Jun}
J.~M. Wills.
\newblock {Zwei S{\ifmmode\ddot{a}\else\"{a}\fi}tze {\ifmmode\ddot{u}\else\"{u}\fi}ber inhomogene diophantische Approximation von Irrationalzehlen}.
\newblock {\em Monatsh. Math.}, 71(3):263--269, 1967.

\bibitem{Wills1968Jun}
J.~M. Wills.
\newblock {Zur simultanen homogenen diophantischen Approximation I}.
\newblock {\em Monatsh. Math.}, 72(3):254--263, 1968.

\end{thebibliography}
\bibliographystyle{abbrv}

\end{document}